\documentclass[reqno,12pt,a4wide]{article}

\usepackage{eucal,enumerate,mathrsfs}
\usepackage{amsmath,amssymb,epsfig,bbm,amsthm}

\numberwithin{equation}{section}

\newtheorem{theorem}{Theorem}[section]

\newtheorem{proposition}[theorem]{Proposition}
\newtheorem{definition}[theorem]{Definition}

\usepackage[normalem]{ulem}  

\usepackage{pdfsync}
\usepackage{eucal,enumerate,mathrsfs}
\usepackage{amsmath,amssymb,epsfig,bbm,amsthm}
\usepackage{hyperref}
\usepackage{ifthen}
\newcommand{\HK}{\mathsf{H\kern-3pt K}}
\newcommand{\xR}{\mathbb R}
\newcommand{\xdif}{\,{\rm d}}
\newcommand{\xLtwo}{{\rm L}^{2}}
\newcommand{\xCinfty}{{\rm C}^{\infty}} 
\newcommand{\xCzero}{{\rm C}^{0}}
\newcommand{\xCone}{{\rm C}^{1}}
\newcommand{\xC}{\mathbb C}

\title{A note on the differentiability of the Hellinger-Kantorovich distances}
\begin{document}

\author{
Florentine Catharina Flei$\ss$ner 
\thanks{Technische Universit\"at M\"unchen  email:
  \textsf{fleissne@ma.tum.de}.
  } }

\date{}

\maketitle

\begin{abstract}
This paper will deal with differentiability properties of the class of Hellinger-Kantorovich distances $\HK_{\Lambda, \Sigma} \ (\Lambda, \Sigma > 0)$ which was recently introduced on the space $\mathcal{M}(\xR^d)$ of finite nonnegative Radon measures. The $\mathscr{L}^1$-a.e.-differentiability of 
\begin{equation*}
t\mapsto \HK_{\Lambda, \Sigma}(\mu_t, \nu)^2,
\end{equation*}
for $\nu\in\mathcal{M}(\xR^d)$ and absolutely continuous curves $(\mu_t)_t$ in $(\mathcal{M}(\xR^d),\HK_{\Lambda, \Sigma})$, will be examined and the corresponding derivatives will be computed. The characterization of absolutely continuous curves in $(\mathcal{M}(\xR^d), \HK_{\Lambda, \Sigma})$ will be refined. 
\end{abstract}

\section{Introduction}
Recently, a new class of distances on the space $\mathcal{M}(\xR^d)$ of finite nonnegative Radon measures was established by three independent teams \cite{liero2016optimal, liero2018optimal, kondratyev2016new, chizat2018interpolating, chizat2018unbalanced}. We will follow the presentation of these distances by Liero, Mielke and Savar\'e \cite{liero2016optimal, liero2018optimal} who named it \textit{Hellinger-Kantorovich distances}. The class of Hellinger-Kantorovich distances $\HK_{\Lambda, \Sigma} \ (\Lambda, \Sigma > 0)$ is based on the conversion of one measure into another one (possibly having different total mass) by means of transport and creation / annihilation of mass. The parameters $\Lambda$ and $\Sigma$ serve as weightings of the transport part and the mass creation/annihilation part respectively. To be more precise, the square $\HK_{\Lambda, \Sigma}(\mu_1, \mu_2)^2$ of the Hellinger-Kantorovich distance $\HK_{\Lambda, \Sigma}$ between two measures $\mu_1, \mu_2 \in\mathcal{M}(\xR^d)$ on $\xR^d$ corresponds to      
\begin{equation}\label{eq: LET}
\min \Big\{\sum_{i=1}^{2}{\frac{4}{\Sigma}\int_{\xR^d}{(\sigma_i\log \sigma_i - \sigma_i + 1)\xdif \mu_i}} + \int_{\xR^d\times \xR^d}{{\sf c}_{\Lambda, \Sigma}(|x_1-x_2|)\xdif \gamma}: \ \gamma\in\mathcal{M}(\xR^d\times\xR^d), \ \gamma_i\ll\mu_i \Big\},
\end{equation}  
with entropy cost functions  $\frac{4}{\Sigma}(\sigma_i\log \sigma_i - \sigma_i + 1)$,
\begin{equation}
\sigma_i := \frac{\xdif \gamma_i}{\xdif\mu_i} \quad (\gamma_i \text{ i-th marginal of } \gamma), 
\end{equation}
and transportation cost function
\begin{equation} 
{\sf c}_{\Lambda, \Sigma}({\sf d}):= 
\begin{cases}
-\frac{8}{\Sigma}\log(\cos(\sqrt{\Sigma/(4\Lambda)} {\sf d})) &\text{ if } {\sf d} < \pi\sqrt{\Lambda/\Sigma},\\
+\infty &\text{ if } {\sf d} \geq \pi\sqrt{\Lambda/\Sigma}.
\end{cases}
\end{equation} 
There exists an optimal plan $\gamma$ for the Logarithmic Entropy-Transport problem \eqref{eq: LET} (cf. Thm. 3.3 in \cite{liero2018optimal}), and if $\mu_1$ is absolutely continuous with respect to the Lebesgue measure and $\gamma$ is such optimal plan, then there exists a Borel optimal transport mapping $t: \xR^d\to\xR^d$ so that $\gamma$ takes the form
\begin{equation*}
\gamma \ = \ (I\times t)_{\#} \gamma_{1} \ = \ (I\times t)_{\#}\sigma_{1}\mu_1  
\end{equation*}
(cf. Thm. 4.5 in \cite{Gangbo-McCann96} and Thm. 6.6 in \cite{liero2018optimal}). We refer the reader to (\cite{liero2018optimal}, Cor. 7.14, Thms. 7.17 and 7.20) for the proofs that $\HK_{\Lambda, \Sigma}$ defined via the Logarithmic Entropy-Transport problem \eqref{eq: LET} indeed represents a distance on the space of finite nonnegative Radon measures and that $(\mathcal{M}(\xR^d), \HK_{\Lambda, \Sigma})$ is a complete metric space. Furthermore, the Hellinger-Kantorovich distance $\HK_{\Lambda, \Sigma}$ metrizes the weak topology on $\mathcal{M}(\xR^d)$ in duality with continuous and bounded functions (cf. Thm. 7.15 in \cite{liero2018optimal}) and can be interpreted as weighted infimal convolution of the Kantorovich-Wasserstein distance and the Hellinger-Kakutani distance. A representation formula \`a la Benamou-Brenier which can be proved for $\HK_{\Lambda, \Sigma}$ (cf. (\cite{liero2018optimal}, Thm. 8.18; \cite{liero2016optimal}, Thm. 3.6(v))) justifies this interpretation: 
\begin{equation}
\HK_{\Lambda, \Sigma}(\mu_1, \mu_2)^2 = \min\Big\{\int_0^1{\int_{\xR^d}{(\Lambda|v_t|^2 + \Sigma|w_t|^2)\xdif \mu_t}\xdif t}: \ \mu_1 \stackrel{(\mu, v,w)}{\rightsquigarrow}\mu_2 \Big\}
\end{equation}    
where $\mu_1\stackrel{(\mu, v, w)}{\rightsquigarrow}\mu_2$ means that $\mu: [0, 1]\to\mathcal{M}(\xR^d)$ is a continuous curve connecting $\mu(0)=\mu_1$ and $\mu(1)=\mu_2$ and satisfying the continuity equation with reaction 
\begin{equation}\label{eq: continuity equation with reaction}
\partial_t\mu_t = -\Lambda\mathrm{div}(v_t\mu_t) + \Sigma w_t\mu_t,
\end{equation}
governed by Borel functions $v: (0,1)\times\xR^d\to\xR^d$ and $w:(0,1)\times\xR^d\to\mathbb{R}$ with
\begin{equation}\label{eq: v,w}
\int_0^1{\int_{\xR^d}{(\Lambda|v_t|^2 + \Sigma|w_t|^2)\xdif \mu_t}\xdif t} < +\infty, 
\end{equation}
in duality with $\xCinfty$-functions with compact support in $(0.1)\times \xR^d$, i.e.
\begin{equation}\label{eq: continuity equation with reaction in duality}
\int_0^1{\int_{\xR^d}{(\partial_t\psi(t,x) + \Lambda\langle\nabla\psi(t,x), v(t,x)\rangle + \Sigma \psi(t,x)w(t,x))\xdif \mu_t(x)}\xdif t} \ = \ 0 
\end{equation}
for all $\psi\in\xCinfty_c((0,1)\times \xR^d)$. 

The class of such continuous curves $\mu$ satisfying (\eqref{eq: continuity equation with reaction}, \eqref{eq: v,w}) for some Borel vector field $(v,w)$ coincides with the class of \textit{absolutely continuous} curves $(\mu_t)_{t\in[0,1]}$ in $(\mathcal{M}(\xR^d), \HK_{\Lambda, \Sigma})$ \textit{ with square-integrable metric derivatives } 
(cf. Thms. 8.16 and 8.17 in \cite{liero2018optimal}, see Sect. \ref{sec: 2} in this paper). 

In order to deepen our understanding of a distance, it is always worth studying its differentiability along absolutely continuous curves (e.g. see Chap. 8 in \cite{AmbrosioGigliSavare05} for the corresponding analysis of the Kantorovich-Wasserstein distance on the space of Borel probability measures with finite second order moments). The present paper addresses this issue for the class of Hellinger-Kantorovich distances on the space of finite nonnegative Radon measures. Clearly, if $(\mu_t)_{t\in[0,1]}$ is an absolutely continuous curve in $(\mathcal{M}(\xR^d), \HK_{\Lambda, \Sigma})$ and $\nu\in\mathcal{M}(\xR^d)$, then the mapping
\begin{equation}\label{eq: map}
t\mapsto \HK_{\Lambda, \Sigma}(\mu_t, \nu)^2
\end{equation} 
is $\mathscr{L}^1$-a.e. differentiable. A natural question that arises is the one of the concrete form of the corresponding derivatives. We will answer this question for absolutely continuous curves with square-integrable metric derivatives (for which such characterization \eqref{eq: continuity equation with reaction} is available), refine that characterization by providing more information on $(v,w)$ (see Prop. \ref{prop: refinement}) and determine 
\begin{equation}\label{eq: map derivative}
\frac{\xdif}{\xdif t} \HK_{\Lambda, \Sigma}(\mu_t, \nu)^2
\end{equation}
at $\mathscr{L}^1$-a.e. $t\in[0,1]$ (see Thm. \ref{thm: main theorem}). This piece of work can be viewed as continuation of Sect. 2 in the author's paper \cite{fleissnerMMHK} constituting a starting point for the study of differentiability properties of the Hellinger-Kantorovich distances. Therein, we identified elements of the Fr\'echet subdifferential of mappings
\begin{equation*}
t\mapsto -\HK_{\Lambda, \Sigma}((I+tv)_{\#}(1+tR)^2\mu_0, \nu)^2
\end{equation*}
at $t=0$, for $\mu_0, \nu\in\mathcal{M}(\xR^d)$ and bounded Borel functions $v: \xR^d\to\xR^d$ and $R: \xR^d\to\mathbb{R}$. That subdifferential calculus was an essential ingredient for our Minimizing Movement approach to a class of scalar reaction-diffusion equations \cite{fleissnerMMHK} substantiating their gradient-flow-like structure in the space of finite nonnegative Radon measures endowed with the Hellinger-Kantorovich distance $\HK_{\Lambda, \Sigma}$.

The proof in \cite{liero2018optimal} that absolutely continuous curves in $(\mathcal{M}(\mathbb H), \HK_{\Lambda, \Sigma})$ with square-integrable metric derivatives are characterized via (\eqref{eq: continuity equation with reaction}, \eqref{eq: v,w}) was carried out only for $\mathbb{H}=\xR^d$, endowed with usual scalar product $\langle\cdot, \cdot\rangle$ and norm $|\cdot|:=\sqrt{\langle\cdot,\cdot\rangle}$, but according to a comment at the beginning of Sect. 8.5 in \cite{liero2018optimal}, it should be possible to prove such characterization result in a more general setting. We would like to remark that also our computation of the derivatives \eqref{eq: map derivative} may be adapted for general separable Hilbert spaces $\mathbb H$.

Our plan for the paper is to give an equivalent characterization of the Hellinger-Kantorovich distances in Sect. \ref{sec: 1} and to perform the computation of the derivatives \eqref{eq: map derivative} in Sect. \ref{sec: 2}.

\section{Optimal transportation on the cone}\label{sec: 1}

According to (\cite{liero2016optimal}, Sect. 3) and (\cite{liero2018optimal}, Sect. 7), the Logarithmic Entropy-Transport problem \eqref{eq: LET} translates into a problem of optimal transportation on the geometric cone $\mathfrak{C}$ on $\xR^d$, see \eqref{eq: cone 1}, \eqref{eq: cone} below. The fact that all the information on transport of mass and creation / annihilation of mass according to \eqref{eq: LET} lies in a pure transportation problem has proved extremely useful for the analysis of $\HK_{\Lambda, \Sigma}$ in \cite{liero2018optimal} and for our subdifferential calculus in \cite{fleissnerMMHK}. 

The geometric cone is defined as the quotient space 
\begin{equation}
\mathfrak{C}:= \xR^d\times [0, +\infty)/\sim
\end{equation}
with 
\begin{equation}
(x_1, r_1) \sim (x_2, r_2) \quad \Leftrightarrow \quad r_1 = r_2 = 0 \text{ or } r_1=r_2, \ x_1 = x_2
\end{equation}
and is endowed with a class of distances ${\sf d}_{\mathfrak{C}, \Lambda, \Sigma} \ (\Lambda, \Sigma > 0)$. The vertex $\mathfrak{o}$ (for $r=0$) and $[x,r]$ (for $x\in \xR^d$ and $r>0$) denote the corresponding equivalence classes 
 and  
\begin{equation}
{\sf d}_{\mathfrak{C}, \Lambda, \Sigma}([x_1,r_1],[x_2, r_2])^2:= 
\frac{4}{\Sigma}\Big(r_1^2 + r_2^2 - 2 r_1 r_2 \cos\Big(\Big(\sqrt{\Sigma/4\Lambda}\ |x_1 - x_2|\Big)\wedge\pi\Big)
\end{equation} 
(where $\mathfrak{o}$ is identified with $[\bar{x}, 0]$ for some $\bar{x}\in \xR^d$). The distance ${\sf d}_{\mathfrak{C}, \Lambda, \Sigma}$ gives rise to an optimal transport problem on the cone and therewith to an extended quadratic Kantorovich-Wasserstein distance $\mathcal{W}_{\mathfrak{C}, \Lambda, \Sigma}$ on the space $\mathcal{M}_2(\mathfrak{C})$ of finite nonnegative Radon measures on $\mathfrak{C}$ with finite second order moments, i.e. $\int_{\mathfrak{C}}{{\sf d}_{\mathfrak{C}, \Lambda, \Sigma}([x,r],\mathfrak{o})^2\xdif \alpha([x,r])} < +\infty$. The extended Kantorovich-Wasserstein distance $\mathcal{W}_{\mathfrak{C}, \Lambda, \Sigma}(\alpha_1, \alpha_2)$ between two measures $\alpha_1, \alpha_2\in\mathcal{M}_2(\mathfrak{C})$ is equal to $+\infty$ if $\alpha_1(\mathfrak{C})\neq\alpha_2(\mathfrak{C})$ and is given by
\begin{equation}\label{eq: Wasserstein on the cone}
\mathcal{W}_{\mathfrak{C}, \Lambda, \Sigma}(\alpha_1, \alpha_2)^2 := 
\min\Big\{\int_{\mathfrak{C}\times\mathfrak{C}}{{\sf d}_{\mathfrak{C}, \Lambda, \Sigma}([x_1,r_1],[x_2,r_2])^2\xdif \beta} \ | \ \beta\in M(\alpha_1, \alpha_2)\Big\}
\end{equation}
if $\alpha_1(\mathfrak{C}) = \alpha_2(\mathfrak{C})$, 
with $M(\alpha_1, \alpha_2)$ being the set of finite nonnegative Radon measures on $\mathfrak{C}\times \mathfrak{C}$ whose first and second marginals coincide with $\alpha_1$ and $\alpha_2$. 
Every measure $\alpha\in\mathcal{M}_2(\mathfrak{C})$ on the cone is assigned a measure $\mathfrak{h}\alpha\in\mathcal{M}(\xR^d)$ on $\xR^d$, 
\begin{equation}\label{eq: homogenous marginal}
\mathfrak{h}\alpha := {\sf x}_{\#}({\sf r}^2\alpha), 
\end{equation}
with 
$({\sf x}, {\sf r}): \mathfrak{C}\to \xR^d\times [0, +\infty)$ defined as 
\begin{equation}\label{eq: homogenous marginal 2}
({\sf x}, {\sf r})([x, r]) := (x,r) \text{ for } [x,r]\in\mathfrak{C}, \ r > 0, \ ({\sf x}, {\sf r})(\mathfrak{o}):=(\bar{x},0), 
\end{equation}
which means $\int_{\xR^d}{\phi(x)\xdif (\mathfrak{h}\alpha)} = \int_{\mathfrak{C}}{{\sf r}^2\phi({\sf x})\xdif \alpha}$ for all continuous and bounded functions $\phi: \xR^d\to\mathbb{R}$ (short $\phi\in\xCzero_b(\xR^d)$).
Please note that the mapping $\mathfrak{h}: \mathcal{M}_2(\mathfrak{C}) \to \mathcal{M}(\xR^d)$ is not injective. 

Now, an equivalent characterization of the Hellinger-Kantorovich distance $\HK_{\Lambda, \Sigma}$ is given by the transportation problems
\begin{eqnarray}
\HK_{\Lambda, \Sigma}(\mu_1, \mu_2)^2 
\ = \ \min\Big\{\mathcal W_{\mathfrak{C}, \Lambda, \Sigma}(\alpha_1, \alpha_2)^2 \ \Big| \ \alpha_i\in\mathcal{M}_2(\mathfrak{C}), \  \mathfrak{h}\alpha_i = \mu_i \Big\} \label{eq: cone 1} \\
\ = \ \min\Big\{\mathcal W_{\mathfrak{C}, \Lambda, \Sigma}(\alpha_1, \alpha_2)^2 + \frac{4}{\Sigma}\sum_{i=1}^{2}{(\mu_i - \mathfrak{h}\alpha_i)(\xR^d)}  \Big| \  \alpha_i\in\mathcal{M}_2(\mathfrak{C}), \ \mathfrak{h}\alpha_i \leq \mu_i \Big\}, \label{eq: cone}
\end{eqnarray}
cf. Probl. 7.4, Thm. 7.6, Lem. 7.9, Thm. 7.20 in \cite{liero2018optimal}. Every solution $\gamma\in\mathcal{M}(\xR^d\times \xR^d)$ to the Logarithmic Entropy-Transport problem \eqref{eq: LET} induces a solution $\beta\in\mathcal{M}(\mathfrak{C}\times\mathfrak{C})$ to (\eqref{eq: cone}, \eqref{eq: Wasserstein on the cone}): if $\gamma$ is an optimal plan for \eqref{eq: LET} with Lebesgue decompositions \begin{footnote}{according to Lem. 2.3 in \cite{liero2018optimal}, there exist Borel functions $\rho_i: \xR^d\to [0, +\infty)$ and nonnegative finite Radon measures $\mu_i^\bot\in\mathcal{M}(\xR^d), \ \mu_i^\bot \bot \gamma_i,$ so that \eqref{eq: Lebesgue decom} holds good} 
\end{footnote}
\begin{equation}\label{eq: Lebesgue decom} 
\mu_i = \rho_i\gamma_i + \mu_i^\bot,
\end{equation} 
then \begin{equation}\label{eq: induced solution on the cone}
\beta := ([x_1, \sqrt{\rho_1(x_1)}], [x_2, \sqrt{\rho_2(x_2)}])_{\#}\gamma \ \in \mathcal{M}(\mathfrak{C}\times\mathfrak{C})
\end{equation}
is an optimal plan for the transport problem (\eqref{eq: cone}, \eqref{eq: Wasserstein on the cone}) (cf. (\cite{liero2018optimal}, Thm. 7.20(iii))). Furthermore, if $\beta\in\mathcal{M}(\mathfrak{C}\times\mathfrak{C})$ is a solution to (\eqref{eq: cone}, \eqref{eq: Wasserstein on the cone}) or a solution to (\eqref{eq: cone 1}, \eqref{eq: Wasserstein on the cone}) (which exists by (\cite{liero2018optimal}, Thm. 7.6)), then 
\begin{equation}\label{eq: transport}
\beta\Big(\Big\{([x_1, r_1], [x_2,r_2])\in\mathfrak{C}\times\mathfrak{C}: \ r_1, r_2 > 0, \ |x_1 - x_2| > \pi\sqrt{\Lambda/\Sigma}\Big\}\Big)= 0,
\end{equation}  
(cf. (\cite{liero2018optimal}, Lem. 7.19)). 

Finally, we show how to construct geodesics in $(\mathfrak{C}, {\sf d}_{\mathfrak{C}, \Lambda, \Sigma})$ (cf. Sect. 8.1 in \cite{liero2018optimal}) as they will play an important role in our analysis of \eqref{eq: map derivative}. We suppose that $|x_1 - x_2| \leq \pi\sqrt{\Lambda/\Sigma}, \ r_1, r_2 > 0,$ and search for functions $\mathcal{R}: [0, 1]\to [0, +\infty)$ and $\theta: [0, 1]\to [0, 1]$ so that the curve $\eta: [0,1]\to\mathfrak{C}$ defined as $\eta(s):=[x_1 + \theta(s)(x_2-x_1), \mathcal{R}(s)]$ is a (constant speed) geodesic connecting $[x_1, r_1]$ and $[x_2, r_2]$, which means ${\sf d}_{\mathfrak{C}, \Lambda, \Sigma}(\eta(s), \eta(t)) = |s-t| {\sf d}_{\mathfrak{C}, \Lambda, \Sigma}([x_1,r_1], [x_2,r_2])$ for all $s, t\in[0,1]$. If $x_1 = x_2$, we set $\theta\equiv 0$.  We note that 
\begin{equation}\label{eq: z1}
{\sf d}_{\mathfrak{C}, \Lambda, \Sigma}(\eta(s), \eta(t))^2 \ = \ 
|z(s) - z(t)|^2_{\xC},
\end{equation}
where $z: [0, 1]\to\xC$ is the curve in the complex plane $\xC$ defined as 
\begin{equation}\label{eq: z}
z(s):=\frac{2}{\sqrt{\Sigma}}\mathcal{R}(s)\exp\Big(i\theta(s)\sqrt{\Sigma/4\Lambda} \ |x_1 - x_2|\Big),
\end{equation}
and $|\cdot|_\xC$ denotes the absolute value for complex numbers. Thus, if $z$ is a geodesic in the complex plane between $z_1:=\frac{2}{\sqrt{\Sigma}} r_1$ and $z_2:=\frac{2}{\sqrt{\Sigma}}r_2\exp\Big(i\sqrt{\Sigma/4\Lambda} \ |x_1 - x_2|\Big)$, i.e. 
\begin{equation}
z(s) = z_1 + s(z_2-z_1) \quad \text{ for all } s\in[0,1],
\end{equation}
then $\eta$ is a geodesic in $(\mathfrak{C}, {\sf d}_{\mathfrak{C}, \Lambda, \Sigma})$ between $[x_1, r_1]$ and $[x_2, r_2]$. This condition yields an appropriate choice for $\mathcal{R}: [0, 1] \to [0, +\infty)$ and $\theta: [0,1]\to[0,1]$, and it is not difficult to see that they are both smooth functions, their first derivatives satisfy
\begin{equation}\label{eq: first derivatives R and theta}
\frac{4}{\Sigma}(\mathcal{R}'(s))^2 + \frac{1}{\Lambda}\mathcal{R}(s)^2(\theta'(s))^2|x_1 - x_2|^2 \ = \ {\sf d}_{\mathfrak{C}, \Lambda, \Sigma}([x_1, r_1], [x_2, r_2])^2 \quad \text{ for all } s\in(0,1),
\end{equation} 
and they are right differentiable at $s=0$. 
We obtain a geodesic from $[x_1, r_1]$ to the vertex $\mathfrak{o}$ by setting $\theta\equiv 0$ and $\mathcal{R}(s):=(1-s)r_1$ and identifying $\mathfrak{o}$ with $[x_1, 0]$. Also in this case, \eqref{eq: first derivatives R and theta} holds good. 

\section{Differentiability results}\label{sec: 2}

We fix $\Lambda, \Sigma > 0$ and examine the behaviour of the Hellinger-Kantorovich distance $\HK_{\Lambda, \Sigma}$ along absolutely continuous curves. 

Let $(\mu_t)_{t\in[0,1]}$ be an absolutely continuous curve in $(\mathcal{M}(\xR^d), \HK_{\Lambda, \Sigma})$ with square-integrable metric derivative, i.e. the limit
\begin{equation}\label{eq: metric derivative}
|\mu_t'| := \mathop{\lim}_{h\to 0} \frac{\HK_{\Lambda, \Sigma}(\mu_{t+h}, \mu_t)}{|h|}
\end{equation}
exists for $\mathscr{L}^1$-a.e. $t\in (0,1)$, the function $t \mapsto |\mu_t'|$ which is called \textit{metric derivative} of $(\mu_t)_t$ belongs to $\xLtwo((0,1))$ and  
\begin{equation}\label{eq: absolutely continuous}
\HK_{\Lambda, \Sigma}(\mu_s,\mu_t) \leq \int^{t}_{s}{|\mu_r'| \xdif r} \quad \quad \text{ for all } 0\leq s \leq t \leq 1 
\end{equation}
(cf. Def. 1.1.1 and Thm. 1.1.2 in \cite{AmbrosioGigliSavare05}). According to Thms. 8.16 and 8.17 in \cite{liero2018optimal}, there exists a Borel vector field $(v,w): (0,1)\times \xR^d \to \xR^d\times\mathbb{R}$ so that the continuity equation with reaction
\begin{equation}\label{eq: cewr}
\partial_t\mu_t = -\Lambda\mathrm{div}(v_t\mu_t) + \Sigma w_t\mu_t
\end{equation}
($v_t:= v(t, \cdot), \ w_t:= w(t, \cdot)$) holds good, in duality with $\xCinfty$-functions with compact support in $(0,1)\times \xR^d$ (see \eqref{eq: continuity equation with reaction in duality}), and 
\begin{equation}\label{eq: vw metric derivative} 
\int_{\xR^d}{(\Lambda|v_t|^2 + \Sigma|w_t|^2)\xdif \mu_t} \ = \ |\mu_t'|^2 \quad\text{ for } \mathscr{L}^1\text{-a.e. } t\in(0,1). 
\end{equation}
For every $t\in(0,1)$ and $h\in(-t, 1-t)$, there exists a plan $\beta_{t, t+h}\in\mathcal{M}(\mathfrak{C}\times\mathfrak{C})$ which is optimal in the definition of $\HK_{\Lambda, \Sigma}(\mu_t, \mu_{t+h})^2$ according to (\eqref{eq: cone 1}, \eqref{eq: Wasserstein on the cone}) and whose first marginal $\pi^1_{\#}\beta_{t, t+h}$ satisfies
\begin{equation}\label{eq: beta}
\int_{\mathfrak{C}}{\phi([x,r])\xdif (\pi^1_{\#}\beta_{t, t+h})} = \int_{\xR^d}{\phi([x, 1])\xdif \mu_t} + h^2\phi(\mathfrak{o})
\end{equation} 
for all $\phi\in\xCzero_b(\mathfrak{C})$ (cf. Thm. 7.6 and Lem. 7.10 in \cite{liero2018optimal}).

We fix $\nu\in\mathcal{M}(\xR^d)$. It follows from \eqref{eq: absolutely continuous} that
\begin{equation}\label{eq: mapping haupt}
t\mapsto \HK_{\Lambda, \Sigma}(\mu_t, \nu) 
\end{equation}
is an absolutely continuous mapping from $[0,1]$ to $[0, +\infty)$ and thus $\mathscr{L}^1$-a.e. differentiable. 

The plan of this section is as follows. First, Prop. \ref{prop: refinement} will identify $(v_t, w_t)$ as belonging to a particular class of functions. Second, the push-forwards of $\beta_{t, t+h}$ through mappings 
\begin{equation}\label{eq: mapping for push-forward}
(y_1,y_2) \mapsto \Big(({\sf x}(y_1),{\sf r}(y_1)), \Big(\frac{1}{h\Lambda}\mathcal{R}_{y_1,y_2}(s)\theta_{y_1,y_2}'(s)({\sf x}(y_2)-{\sf x}(y_1)), \frac{2}{h\Sigma}\mathcal{R}_{y_1,y_2}'(s)\Big)\Big)
\end{equation} 
from $\Big(\mathfrak{C}\times\mathfrak{C}\Big)\setminus\Big\{([x_1, r_1], [x_2,r_2])\in\mathfrak{C}\times\mathfrak{C}: \ r_1, r_2 > 0, \ |x_1 - x_2| > \pi\sqrt{\Lambda/\Sigma}\Big\}$ to $(\xR^d\times\xR)\times(\xR^d\times\xR)$ will be considered, for $s\in(0,1)$, with $y_i:=[x_i,r_i]$, ${\sf x}$, ${\sf r}$ as in \eqref{eq: homogenous marginal 2}, and 
\begin{equation*}
[0,1]\ni s\mapsto (\theta_{[x_1,r_1], [x_2,r_2]}(s), \mathcal{R}_{[x_1,r_1],[x_2,r_2]}(s)) \in [0,1]\times [0,+\infty)
\end{equation*}
being constructed according to Sect. \ref{sec: 1} (cf. \eqref{eq: z1}-\eqref{eq: first derivatives R and theta}) so that 
\begin{equation}\label{eq: geodesics for push-forward}
s\mapsto [x_1 + \theta_{[x_1,r_1],[x_2,r_2]}(s)(x_2-x_1), \mathcal{R}_{[x_1,r_1],[x_2,r_2]}(s)] \text{ is a geodesic from } [x_1,r_1] \text{ to } [x_2,r_2]. 
\end{equation}
Please recall \eqref{eq: transport} in this context and note that, by \eqref{eq: first derivatives R and theta}, the mappings \eqref{eq: mapping for push-forward} are Borel measurable.
Their second components 
may be interpreted as blow-ups of tangent vectors to geodesics in $(\mathfrak{C}, {\sf d }_{\mathfrak{C}, \Lambda, \Sigma})$ and Prop. \ref{prop: push-forward} will provide information on the limits of the corresponding push-forwards of $\beta_{t,t+h}$ as $h\to0$, linking them to $(v_t,w_t)$. That result will be helpful in studying the $\mathscr{L}^1$-a.e.-differentiability of the mapping \eqref{eq: mapping haupt} and finally, in Thm. \ref{thm: main theorem}, we will determine the derivatives by computing
\begin{equation}\label{eq: derivatives haupt}
\frac{\xdif}{\xdif t} \HK_{\Lambda, \Sigma}(\mu_t, \nu)^2
\end{equation}  
at $\mathscr{L}^1$-a.e. $t\in(0,1)$.  
\\ 

The above notation holds good throughout this section. 
\begin{proposition}\label{prop: refinement}
For $\mathscr{L}^1$-a.e. $t\in(0,1)$, the Borel function $(v_t, w_t)$ belongs to the closure in $\xLtwo(\mu_t, \xR^d\times\mathbb{R})$ of the subspace $\{(\nabla\zeta, \zeta): \ \zeta\in\xCinfty_c(\xR^d)\}$. 
\end{proposition}
Here $(\xLtwo(\mu_t, \xR^d\times\xR), ||\cdot||_{\xLtwo(\mu_t, \xR^d\times\xR)})$ denotes the normed space of all $\mu_t$-measurable functions $(\bar{v}, \bar{w})$ from $\xR^d$ to $\xR^d\times\xR$ satisfying  
\begin{equation}
||(\bar{v},\bar{w})||_{\xLtwo(\mu_t, \xR^d\times\xR)} := \Big(\int_{\xR^d}{(\Lambda|\bar{v}|^2 + \Sigma|\bar{w}|^2)\xdif \mu_t}\Big)^{1/2} < +\infty. 
\end{equation}
\begin{proof}
We construct a Borel vector field $(\tilde{v}, \tilde{w}): (0,1)\times\xR^d\to\xR^d\times\xR$ satisfying \eqref{eq: cewr} so that, for $\mathscr{L}^1$-a.e. $t\in(0,1)$, the function $(\tilde{v}_t,\tilde{w}_t)$ belongs to the closure in $\xLtwo(\mu_t, \xR^d\times\mathbb{R})$ of the subspace $\{(\nabla\zeta, \zeta): \ \zeta\in\xCinfty_c(\xR^d)\}$  and 
\begin{equation}\label{eq: tilde vw metric derivative}
||(\tilde{v}_t,\tilde{w}_t)||_{\xLtwo(\mu_t, \xR^d\times\xR)}^2 = \int_{\xR^d}{(\Lambda|\tilde{v}_t|^2 + \Sigma|\tilde{w}_t|^2)\xdif \mu_t} \ \leq \ |\mu_t'|^2. 
\end{equation} 
We begin the proof with some estimations. 
Let $\phi\in\xCinfty_c(\xR^d)$. It follows from the construction of $\mathcal{R}_{[x_1,r_1],[x_2,r_2]}$ and $\theta_{[x_1,r_1],[x_2,r_2]}$ according to \eqref{eq: z1}-\eqref{eq: first derivatives R and theta} that 
\begin{eqnarray*}
\frac{2}{\Sigma}\frac{\xdif}{\xdif^2 s}\mathcal{R}_{[x_1,r_1],[x_2,r_2]}(s)^2 & = & {\sf d}_{\mathfrak{C}, \Lambda, \Sigma}([x_1,r_1],[x_2,r_2])^2, \\
\Big|\theta_{[x_1,r_1],[x_2,r_2]}''(s)\mathcal{R}_{[x_1,r_1],[x_2,r_2]}(s)^2(x_2-x_1)\Big| &\leq& C_{\Sigma, \Lambda} {\sf d}_{\mathfrak{C},\Lambda, \Sigma}([x_1,r_1],[x_2,r_2])^2, \\
\Big|2\theta_{[x_1,r_1],[x_2,r_2]}'(s)\mathcal{R}_{[x_1,r_1],[x_2,r_2]}(s)\mathcal{R}'_{[x_1,r_1],[x_2,r_2]}(s)(x_2-x_1)\Big| &\leq& C_{\Sigma, \Lambda}{\sf d}_{\mathfrak{C},\Lambda, \Sigma}([x_1,r_1],[x_2,r_2])^2, \\ 
\Big|\frac{\xdif}{\xdif^2 s}\Big[\phi(x_1+\theta_{[x_1,r_1],[x_2,r_2]}(s)(x_2-x_1))\mathcal{R}_{[x_1,r_1],[x_2,r_2]}(s)^2\Big]\Big| &\leq& C_\phi C_{\Sigma,\Lambda}{\sf d}_{\mathfrak{C}, \Lambda, \Sigma}([x_1,r_1],[x_2,r_2])^2,
\end{eqnarray*}
for $s\in(0,1)$, with $C_\phi > 0$ only depending on $\phi$ and $C_{\Sigma, \Lambda}:=2\Sigma + 4\Lambda$; we refer the reader to the proof of Prop. 2.5 in \cite{fleissnerMMHK} for details. With \eqref{eq: first derivatives R and theta} and these estimations on hand, it is straightforward to prove that there exists a constant $C_{\phi,\Lambda, \Sigma} > 0$ only depending on $\phi, \ \Lambda$ and $\Sigma$ so that
\begin{equation}\label{eq: 0}
|\varphi'_{y_1,y_2}(\bar{s}) - \varphi'_{y_1,y_2}(s)| \leq C_{\phi,\Lambda, \Sigma} \ {\sf d}_{\mathfrak{C},\Lambda,\Sigma}(y_1,y_2)^2, 
\end{equation}
\begin{equation}\label{eq: 1}
\Big|\varphi'_{y_1,y_2}(s) - \langle\nabla\phi(x_1), \theta'_{y_1, y_2}(s)(x_2-x_2)\rangle\mathcal{R}_{y_1,y_2}(s)^2 + 2\phi(x_1)\mathcal{R}'_{y_1,y_2}(s)\mathcal{R}_{y_1,y_2}(s)\Big| \leq C_{\phi,\Lambda, \Sigma} \ {\sf d}_{\mathfrak{C},\Lambda,\Sigma}(y_1,y_2)^2
\end{equation}
and 
\begin{equation}\label{eq: 2}
\Big|\Big(\langle\nabla\phi(x_1), \theta'_{y_1, y_2}(s)(x_2-x_2)\rangle\mathcal{R}_{y_1,y_2}(s) + 2\phi(x_1)\mathcal{R}'_{y_1,y_2}(s)\Big)\Big(\mathcal{R}_{y_1,y_2}(s)-r_1\Big)\Big| \leq C_{\phi,\Lambda, \Sigma} \ {\sf d}_{\mathfrak{C},\Lambda,\Sigma}(y_1,y_2)^2
\end{equation}
for all $s, \bar{s} \in(0,1)$, with $y_i:=[x_i,r_i], \ \varphi_{y_1,y_2}(s):=\phi(x_1+\theta_{[x_1,r_1],[x_2,r_2]}(s)(x_2-x_1))\mathcal{R}_{[x_1,r_1],[x_2,r_2]}(s)^2$. 

Now, let $t\in(0,1)$ so that the limit \eqref{eq: metric derivative} exists and $\mathfrak{C}_\mathfrak{o} := \mathfrak{C}\setminus\{\mathfrak{o}\}$. By applying \eqref{eq: transport}, \eqref{eq: 1}, \eqref{eq: 2}, \eqref{eq: beta}, H\"older's inequality and \eqref{eq: first derivatives R and theta}, we obtain
\begin{align*}
&\Big|\int_{\xR^d}{\phi\xdif \mu_{t+h}} - \int_{\xR^d}{\phi\xdif\mu_t}\Big| = \Big|\int_{\mathfrak{C}\times\mathfrak{C}}{(\phi(x_2)r_2^2 - \phi(x_1)r_1^2)\xdif\beta_{t,t+h}}\Big|  \leq \int_{\mathfrak{C}\times\mathfrak{C}}{\int_0^1{|\varphi'_{y_1,y_2}(s)| 
\xdif s}\xdif \beta_{t,t+h}} \leq \\
& \int_{\mathfrak{C}_{\mathfrak{o}}\times\mathfrak{C}}{\int_0^1{\Big|\langle\nabla\phi(x_1),\theta'_{[x_1,r_1],[x_2,r_2]}(s)(x_2-x_1)\rangle \mathcal{R}_{[x_1,r_1],[x_2,r_2]}(s)  + 2\phi(x_1)\mathcal{R}'_{[x_1,r_1],[x_2,r_2]}(s)\Big|\xdif s}\xdif\beta_{t,t+h}}  \\
&\quad \quad+ 2C_{\phi,\Lambda, \Sigma}\HK_{\Lambda, \Sigma}(\mu_{t},\mu_{t+h})^2 \ \leq  \\
&\Big(\int_{\mathfrak{C}_\mathfrak{o}}{\Big(\Lambda|\nabla\phi|^2 + \Sigma\phi^2\Big)\xdif(\pi^1_{\#}\beta_{t,t+h})}\Big)^{1/2} \Big(\int_{\mathfrak{C}_\mathfrak{o}\times\mathfrak{C}}{\int_0^1{\Big(\frac{1}{\Lambda}\mathcal{R}^2(\theta')^2|x_2-x_1|^2 + \frac{4}{\Sigma}(\mathcal{R}')^2\Big)\xdif s}\xdif \beta_{t,t+h}}\Big)^{1/2} \\ & \quad \quad + 2C_{\phi,\Lambda, \Sigma}\HK_{\Lambda, \Sigma}(\mu_{t},\mu_{t+h})^2 \ \leq \\ 
 & \leq \ ||(\nabla\phi, \phi)||_{\xLtwo(\mu_t, \xR^d\times\xR)} \HK_{\Lambda, \Sigma}(\mu_t,\mu_{t+h}) \ + 2C_{\phi,\Lambda, \Sigma}\HK_{\Lambda, \Sigma}(\mu_{t},\mu_{t+h})^2
\end{align*}
and thus, 
\begin{equation}\label{eq: L}
\limsup_{h\to0}\frac{1}{|h|}\Big|\int_{\xR^d}{\phi\xdif \mu_{t+h}} - \int_{\xR^d}{\phi\xdif\mu_t}\Big| \ \leq \  ||(\nabla\phi, \phi)||_{\xLtwo(\mu_t, \xR^d\times\xR)}|\mu_t'|.
\end{equation}
At this point, we may follow the proof of Thm. 8.3.1 in \cite{AmbrosioGigliSavare05}. Therein, a similar characterization of absolutely continuous curves in the space of Borel probability measures with finite second order moments, endowed with the Kantorovich-Wasserstein distance, was given by solving a suitable minimum problem. We adapt that approach. Let $\mu\in\mathcal{M}((0,1)\times\xR^d)$ be defined by 
\begin{equation*}
\int_{(0,1)\times\xR^d}{\psi(t,x)\xdif\mu(t,x)} \ = \ \int_0^1{\int_{\xR^d}{\psi(t,x)\xdif\mu_t(x)}\xdif t}
\end{equation*}
for all $\psi\in\xCzero_b((0,1)\times\xR^d)$, and let $(\xLtwo(\mu,\xR^d\times\xR), ||\cdot||_{\xLtwo(\mu, \xR^d\times\xR)})$ denote the normed space of all $\mu$-measurable vector fields $(\hat{v},\hat{w})$ from $(0,1)\times\xR^d$ to $\xR^d\times\xR$ satisfying
\begin{equation}
||(\hat{v}_t,\hat{w}_t)||_{\xLtwo(\mu, \xR^d\times\xR)} := \Big(\int_0^1{\int_{\xR^d}{(\Lambda|\hat{v}_t|^2 + \Sigma|\hat{w}_t|^2)\xdif \mu_t}\xdif t}\Big)^{1/2} < +\infty. 
\end{equation}
An application of \eqref{eq: L}, Fatou's Lemma, H\"older's inequality and Hahn-Banach Theorem shows that there exists a unique bounded linear functional $L$ defined on the closure $\mathcal{V}$ in $\xLtwo(\mu, \xR^d\times\xR)$ of the subspace $\{(\nabla\zeta,\zeta): \ \zeta\in\xCinfty_c((0,1)\times\xR^d)\}$, satisfying
\begin{equation}
L((\nabla\zeta, \zeta)) := -\int_0^1{\int_{\xR^d}{\partial_t\zeta(t,x)\xdif \mu_t}\xdif t} \quad \text{for all $\zeta\in\xCinfty_c((0,1)\times\xR^d)$.}
\end{equation}
We consider the minimum problem
\begin{equation}\label{eq: min}
\min \Big\{\frac{1}{2}||(\hat{v},\hat{w})||^2_{\xLtwo(\mu, \xR^d\times\xR)} - L((\hat{v},\hat{w})): \ (\hat{v},\hat{w})\in\mathcal{V}\Big\}.
\end{equation}
The same argument as in the proof of Thm. 8.3.1 in \cite{AmbrosioGigliSavare05} proves that the unique solution $(\tilde{v},\tilde{w})$ to \eqref{eq: min} (which clearly exists) satisfies \eqref{eq: cewr} and, for $\mathscr{L}^1$-a.e. $t\in(0,1)$, the function $(\tilde{v}_t,\tilde{w}_t)$ belongs to the closure in $\xLtwo(\mu_t, \xR^d\times\mathbb{R})$ of the subspace $\{(\nabla\zeta, \zeta): \ \zeta\in\xCinfty_c(\xR^d)\}$  and \eqref{eq: tilde vw metric derivative} holds good. By Thm. 8.17 in \cite{liero2018optimal}, for every Borel vector field $(\hat{v}, \hat{w})\in\xLtwo(\mu,\xR^d\times\xR)$ satisfying the continuity equation with reaction \eqref{eq: cewr} the opposite inequality holds good, i.e.
\begin{equation*}
\int_{\xR^d}{(\Lambda|\hat{v}_t|^2 + \Sigma|\hat{w}_t|^2)\xdif \mu_t} \ \geq \ |\mu_t'|^2 \quad\text{ for } \mathscr{L}^1\text{-a.e.} t\in(0,1). 
\end{equation*} 
It follows from this and from the strict convexity of $||\cdot||^2_{\xLtwo(\mu_t, \xR^d\times\xR)}$ that the Borel vector field $(\tilde{v},\tilde{w})$ solves (\eqref{eq: cewr}, \eqref{eq: vw metric derivative}) and that it coincides $\mathscr{L}^1$-a.e. with any other vector field solving (\eqref{eq: cewr}, \eqref{eq: vw metric derivative}). This completes the proof of Prop. \ref{prop: refinement}.  
\end{proof}

\begin{definition}\label{def: N}
\textnormal{Let $\mathcal{D}(\xR^d)$ be a countable subset of $\xCinfty_c(\xR^d)$ so that every function in $\xCinfty_c(\xR^d)$ can be approximated in the $\xCone$-norm by a sequence of functions in $\mathcal{D}(\xR^d)$.}

\textnormal{We define $\mathcal{N}$ as the set of points $t\in(0,1)$ at which the following holds good: 
\begin{enumerate}[(i)]
\item the limit \eqref{eq: metric derivative} exists,
\item $(v_t,w_t)$ belongs to the closure in $\xLtwo(\mu_t, \xR^d\times\mathbb{R})$ of the subspace $\{(\nabla\zeta, \zeta): \ \zeta\in\xCinfty_c(\xR^d)\}$ and satisfies \eqref{eq: vw metric derivative},
\item the mapping
\begin{equation}\label{eq: 50}
t\mapsto \frac{1}{2}\HK_{\Lambda, \Sigma}(\mu_t, \nu)^2
\end{equation}
is differentiable at $t$, 
\item and, for all $\psi\in \mathcal{D}(\xR^d)$, 
\begin{equation}\label{eq: N}
\lim_{h\to 0}\frac{1}{h}\Big(\int_{\xR^d}{\psi\xdif \mu_{t+h}} - \int_{\xR^d}{\psi\xdif \mu_{t}}\Big) = \int_{\xR^d}{(\Lambda\langle\nabla\psi, v_t\rangle + \Sigma \psi w_t)\xdif \mu_t}.
\end{equation}
\end{enumerate}
}
\end{definition}
Please note that $(0,1)\setminus \mathcal{N}$ is an $\mathscr{L}^1$-negligible set;
it follows from \eqref{eq: continuity equation with reaction in duality} that, for fixed $\psi\in\xCinfty_c(\xR^d)$, the mapping $t\mapsto \int_{\xR^d}{\psi\xdif\mu_t}$ is absolutely continuous from $[0,1]$ to $\xR$ and \eqref{eq: N} holds good at $\mathscr{L}^1$-a.e. $t\in(0,1)$. 

We turn to the push-forward $\Delta_{t,h,s}\in\mathcal{M}((\xR^d\times\xR)\times(\xR^d\times\xR))$ of $\beta_{t,t+h}$ through \eqref{eq: mapping for push-forward}, defined by
\begin{eqnarray*}
\int_{(\xR^d\times\xR)\times(\xR^d\times\xR)}{\Phi(y)\xdif \Delta_{t,h,s}}\quad\quad\quad \\ =  \int_{\mathfrak{C}\times\mathfrak{C}}{\Phi\Big(({\sf x},{\sf r})([x_1,r_1]), \Big(\frac{1}{h\Lambda}\mathcal{R}_{[x_1,r_1],[x_2,r_2]}(s)\theta_{[x_1,r_1],[x_2,r_2]}'(s)(x_2-x_1), \frac{2}{h\Sigma}\mathcal{R}_{[x_1,r_1],[x_2,r_2]}'(s)\Big)\Big)\xdif\beta_{t,t+h}}
\end{eqnarray*} 
for all $\Phi\in\xCzero_b((\xR^d\times\xR)\times(\xR^d\times\xR))$. 
\begin{proposition}\label{prop: push-forward} The following holds good for all $t\in\mathcal{N}$. 
\begin{enumerate}[(i)]
\item Let $s\in(0,1)$. Then
\begin{equation}\label{eq: convergence of delta}
\lim_{h\to 0} \int_{(\xR^d\times\xR)\times(\xR^d\times\xR)}{\Phi(y)\xdif \Delta_{t,h,s}} \ = \ \int_{\xR^d}{\Phi((x,1),(v_t(x),w_t(x)))\xdif \mu_t}
\end{equation}
for all continuous functions $\Phi: (\xR^d\times\xR)\times(\xR^d\times\xR)\to\xR$ satisfying the growth condition
\begin{equation}\label{eq: growth condition}
|\Phi((x_1,r_1),(x_2,r_2))| \ \leq \ C\Big(1+|x_2|^2+|r_2|^2\Big) 
\end{equation}
for some $C>0$. 
\item Define $\mathfrak{C}_{t,h}:=\Big\{[x,r]\in\mathfrak{C}\setminus\{\mathfrak{o}\}: \ |v_t(x)| < \frac{1}{\sqrt{|h|}} \text{ and } |w_t(x)| < \frac{2}{\sqrt{|h|}\Sigma}\Big\}$ and $\Xi_{t,h}:\mathfrak{C}\to\mathfrak{C}$, 
\begin{equation}\label{eq: def of Xi}
\Xi_{t,h}([x,r]):= \begin{cases} [x+\Lambda hv_t(x), r(1+\frac{\Sigma}{2}hw_t(x))] &\text{ if } [x, r]\in\mathfrak{C}_{t,h}, \\ [x,r] &\text{ else}. 
\end{cases}
\end{equation}
Let $\chi_{t,h}:=(\Xi_{t,h})_{\#}(\pi^1_{\#}\beta_{t,t+h})$ be the push-forward of the first marginal of $\beta_{t,t+h}$ through $\Xi_{t,h}$, i.e. 
\begin{equation*}
\int_{\mathfrak{C}}{\phi([x,r])\xdif \chi_{t,h}}  =  \int_{\mathfrak{C}}{\phi(\Xi_{t,h}([x,r]))\xdif (\pi^1_{\#}\beta_{t,t+h})}
\end{equation*}
for all $\phi\in\xCzero_b(\mathfrak{C})$. Then
\begin{equation}\label{eq: Xi}
\lim_{h\to 0} \frac{\HK_{\Lambda, \Sigma}(\mu_{t+h}, \mathfrak{h}\chi_{t,h})^2}{h^2} \ = \ 0. 
\end{equation}
\end{enumerate}
\end{proposition}
\begin{proof} We set $Y:=\xR^d\times\xR$. 

(i) Let $t\in\mathcal{N}$ and $s\in(0,1)$. We note that, by \eqref{eq: first derivatives R and theta} and Def. \ref{def: N}(i),  
\begin{equation}\label{eq: prokhorov}
\int_{Y\times Y}{(\Lambda|x_2|^2 + \Sigma|r_2|^2)\xdif \Delta_{t,h,s}((x_1,r_1),(x_2,r_2))} = \frac{\HK_{\Lambda, \Sigma}(\mu_t,\mu_{t+h})^2}{h^2} \to |\mu_t'|^2 \quad\text{ as } h\to0. 
\end{equation}
We may apply Prokhorov's Theorem to any sequence $(\Delta_{t,h_k,s})_{k\in\mathbb{N}}, \ h_k\to 0,$ of measures from the family $(\Delta_{t,h,s})_{h\in(-t,1-t)}\subset \mathcal{M}(Y\times Y)$, since such sequence is bounded and equally tight by \eqref{eq: beta} and \eqref{eq: prokhorov}, and we obtain a subsequence $h_{k_l}\to 0$ and a measure $\Delta\in\mathcal{M}(Y\times Y)$ so that $(\Delta_{t,h_{k_l},s})_{l\in\mathbb{N}}$ converges to $\Delta$ in the weak topology on $\mathcal{M}(Y\times Y)$, in duality with continuous and bounded functions. So let $(\Delta_{t,h_l,s})_{l\in\mathbb{N}} \ (h_l\to0)$ be a convergent sequence with limit measure $\Delta\in\mathcal{M}(Y\times Y)$, i.e.
\begin{equation}\label{eq: convergent}
\lim_{l\to\infty}\int_{Y\times Y}{\Phi(y)\xdif \Delta_{t,h_l,s}} \ = \ \int_{Y\times Y}{\Phi(y)\xdif \Delta}
\end{equation}
for all $\Phi\in\xCzero_b(Y\times Y)$. We want to identify $\Delta$ as $((x,1),(v_t(x),w_t(x)))_{\#}\mu_t$. It is not difficult to infer from \eqref{eq: beta} that the first marginal $\pi^1_{\#}\Delta$ of $\Delta$ coincides with $(x,1)_{\#}\mu_t$, i.e. 
\begin{equation}\label{eq: first marginal Delta}
\int_{Y}{\phi((x,r))\xdif (\pi^1_{\#}\Delta)} = \int_{\xR^d}{\phi((x,1))\xdif \mu_t}
\end{equation} 
for all $\phi\in\xCzero_b(Y)$. Let $\psi\in\mathcal{D}(\xR^d)$. Then \eqref{eq: convergent} also holds good for $\Phi((x_1,r_1),(x_2,r_2)):=\Big[\Lambda\langle\nabla\psi(x_1),x_2\rangle + \Sigma\psi(x_1)r_2\Big]r_1$: Indeed, we have
\begin{equation*}
\lim_{l\to\infty}\int_{Y\times Y}{(\Phi_N) \xdif \Delta_{t,h_l,s}} = \int_{Y\times Y}{(\Phi_N) \xdif \Delta}
\end{equation*} 
for all $N > 0$, with $\Phi_N:= (\Phi\wedge N)\vee(-N)$. Setting $Y_N:=\{(x,r)\in Y: \ |x|+|r| > N\}, \ C_\psi:=\sup_{x\in\xR^d}\{|\nabla\psi(x)|+|\psi(x)|\},$ and applying \eqref{eq: prokhorov},\eqref{eq: beta} and \eqref{eq: first marginal Delta}, we conclude that for every $\epsilon > 0$ there exists $N_\epsilon > 0$ so that 
\begin{equation*}
\int_{Y\times Y_N}{(|x_2|+|r_2|)\xdif \Delta_{t,h_l,s}} + \int_{Y\times Y_N}{(|x_2|+|r_2|)\xdif \Delta} \ \leq \epsilon \quad \text{ for all } N\geq N_\epsilon, \ l\in\mathbb{N},
\end{equation*}
and
\begin{eqnarray*}
&& \limsup_{l\to\infty} \Big|\int_{Y\times Y}{\Phi\xdif \Delta_{t,h_l,s}} - \int_{Y\times Y}{\Phi\xdif \Delta}\Big| \\ &\leq& \limsup_{l\to\infty} \Big|\int_{Y\times Y}{(\Phi_{C_\psi(\Lambda + \Sigma)N_\epsilon})\xdif \Delta_{t,h_l,s}} - \int_{Y\times Y}{\Phi_{C_\psi(\Lambda + \Sigma)N_\epsilon} \xdif \Delta}\Big| \\ && \quad + \ C_\psi(\Lambda+\Sigma)\limsup_{l\to\infty}\int_{Y\times Y_{N_\epsilon}}{(|x_2|+|r_2|)\xdif (\Delta_{t,h_l,s}+\Delta)} \\
&\leq& C_\psi(\Lambda + \Sigma)\epsilon. 
\end{eqnarray*}
Hence, taking \eqref{eq: first marginal Delta} into account, we obtain
\begin{equation}\label{eq: 3}
\lim_{l\to\infty}\int_{Y\times Y}{\Big[\Lambda\langle\nabla\psi(x_1),x_2\rangle + \Sigma\psi(x_1)r_2\Big]r_1\xdif \Delta_{t,h_l,s}} = \int_{Y\times Y}{\Big[\Lambda\langle\nabla\psi(x_1),x_2\rangle + \Sigma\psi(x_1)r_2\Big] \xdif \Delta}.
\end{equation}

It holds that
\begin{eqnarray*}
\int_{\xR^d}{\psi\xdif\mu_{t+h_l}} - \int_{\xR^d}{\psi\xdif\mu_{t}} = \int_{\mathfrak{C}\times\mathfrak{C}}{(\psi(x_2)r_2^2 - \psi(x_1)r_1^2)\xdif\beta_{t,t+h}} \\ = \int_{\mathfrak{C}\times\mathfrak{C}}{\int_0^1{\frac{\xdif}{\xdif s}\Big[\psi(x_1+\theta_{[x_1,r_1],[x_2,r_2]}(s)(x_2-x_1))\mathcal{R}_{[x_1,r_1],[x_2,r_2]}(s)^2\Big] \xdif s}\xdif \beta_{t,t+h_l}}
\end{eqnarray*}
so that \eqref{eq: N}, \eqref{eq: 0}, \eqref{eq: 1}, \eqref{eq: 2}, Def. \ref{def: N}(i) and \eqref{eq: 3} yield
\begin{eqnarray*}
\int_{\xR^d}{(\Lambda\langle\nabla\psi, v_t\rangle + \Sigma\psi w_t)\xdif \mu_t} = \lim_{l\to \infty}\frac{1}{h_l}\Big(\int_{\xR^d}{\psi\xdif \mu_{t+h_l}} - \int_{\xR^d}{\psi\xdif\mu_t}\Big) \\ = \lim_{l\to\infty}\int_{Y\times Y}{\Big[\Lambda\langle\nabla\psi(x_1),x_2\rangle + \Sigma\psi(x_1)r_2\Big]r_1\xdif \Delta_{t,h_l,s}} = \int_{Y\times Y}{\Big[\Lambda\langle\nabla\psi(x_1),x_2\rangle + \Sigma\psi(x_1)r_2\Big] \xdif \Delta}. 
\end{eqnarray*}
According to the Disintegration Theorem (see e.g. Thm. 5.3.1 in \cite{AmbrosioGigliSavare05}) and \eqref{eq: first marginal Delta}, there exists a Borel family of probability measures $(\Delta_{x_1})_{x_1\in\xR^d}\subset \mathcal{M}(Y), \ \Delta_{x_1}(Y) = 1,$ so that
\begin{equation*}
\int_{Y\times Y}{\Phi\xdif \Delta} = \int_{\xR^d}{\Big(\int_Y{\Phi((x_1,1), (x_2,r_2))\xdif \Delta_{x_1}((x_2,r_2))}\Big)\xdif \mu_t(x_1)}
\end{equation*}
for all $\Delta$-integrable maps $\Phi: Y\times Y \to \xR$. We infer from \eqref{eq: prokhorov} that, for $\mu_t$-a.e. $x_1\in\xR^d$, the measure $\Delta_{x_1}$ has finite second order moment and we define the function $(v_\Delta, w_\Delta): \xR^d\to \xR^d\times \xR$ by
\begin{equation}\label{eq: barycentric projection}
v_\Delta(x_1) := \int_Y{x_2\xdif \Delta_{x_1}((x_2,r_2))}, \ w_{\Delta}(x_1):=\int_Y{r_2\xdif\Delta_{x_1}((x_2,r_2))} \quad\text{ for } \mu_t \text{-a.e. } x_1\in\xR^d.
\end{equation}
The function $(v_\Delta, w_\Delta)$ is Borel measurable (cf. (5.3.1) and Def. 5.4.2 in \cite{AmbrosioGigliSavare05}), and 
\begin{eqnarray*}
&& \int_{Y\times Y}{\Big[\Lambda\langle\nabla\psi(x_1),x_2\rangle + \Sigma\psi(x_1)r_2\Big] \xdif \Delta} \\ &=& \int_{\xR^d}{\Big(\int_Y{\Big[\Lambda\langle\nabla\psi(x_1),x_2\rangle + \Sigma\psi(x_1)r_2\Big] \xdif \Delta_{x_1}((x_2,r_2))}\Big)\xdif \mu_t(x_1)} \\ &=& \int_{\xR^d}{(\Lambda\langle\nabla\psi, v_\Delta\rangle + \Sigma \psi w_\Delta)\xdif \mu_t}.
\end{eqnarray*} 
All in all, we have found that
\begin{equation}\label{eq: cewr Delta}
\int_{\xR^d}{(\Lambda\langle\nabla\psi, v_t\rangle + \Sigma \psi w_t)\xdif \mu_t} = \int_{\xR^d}{(\Lambda\langle\nabla\psi, v_\Delta\rangle + \Sigma \psi w_\Delta)\xdif \mu_t}
\end{equation}
for all $\psi\in\mathcal{D}(\xR^d)$. Since every function in $\xCinfty_c(\xR^d)$ can be approximated in the $\xCone$-norm by a sequence of functions in $\mathcal{D}(\xR^d)$ (cf. Def. \ref{def: N}) and, by \eqref{eq: prokhorov} and Def. \ref{def: N}(ii), the functions $v_\Delta, w_\Delta, v_t, w_t$ are square-integrable w.r.t. $\mu_t$, 
\eqref{eq: cewr Delta} holds good for all $\psi\in\xCinfty_c(\xR^d)$ and for all pairs in the $\xLtwo(\mu_t, \xR^d\times\xR)$-closure of $\{(\nabla\zeta,\zeta): \ \zeta\in\xCinfty_c(\xR^d)\}$. It follows from this and from Def. \ref{def: N}(ii) that 
\begin{equation}\label{eq: step 1}
||(v_t,w_t)||_{\xLtwo(\mu_t, \xR^d\times\xR)}^2 \ = \int_{\xR^d}{(\Lambda\langle v_t, v_\Delta\rangle + \Sigma w_t w_\Delta)\xdif \mu_t}. 
\end{equation}
Applying H\"older's inequality to \eqref{eq: step 1}, taking the definition \eqref{eq: barycentric projection} of $v_\Delta, w_\Delta$, Jensen's inequality, \eqref{eq: convergent}, \eqref{eq: prokhorov} and Def. \ref{def: N}(ii) into account, we obtain
\begin{eqnarray}
||(v_t,w_t)||_{\xLtwo(\mu_t, \xR^d\times\xR)} \leq ||(v_\Delta,w_\Delta)||_{\xLtwo(\mu_t, \xR^d\times\xR)} \leq \Big(\int_{Y\times Y}{(\Lambda|x_2|^2 + \Sigma|r_2|^2)\xdif \Delta}\Big)^{1/2} \label{eq: step 2} \leq \\ \leq \lim_{l\to\infty}\Big(\int_{Y\times Y}{(\Lambda|x_2|^2 + \Sigma|r_2|^2)\xdif \Delta_{t,h_l,s}}\Big)^{1/2} =  ||(v_t,w_t)||_{\xLtwo(\mu_t, \xR^d\times\xR)} \label{eq: step 3}
\end{eqnarray}
so that, in fact, equality holds good everywhere in \eqref{eq: step 2} and \eqref{eq: step 3}. We infer from this and from \eqref{eq: step 1} that
\begin{equation*}
||(v_t,w_t) - (v_\Delta, w_\Delta)||_{\xLtwo(\mu_t, \xR^d\times\xR)} = 0
\end{equation*}
which means 
\begin{equation}\label{eq: step 4}
v_t(x) = v_\Delta(x) \text{ and } w_t(x) = w_\Delta(x) \quad\quad \text{ for } \mu_t\text{-a.e. } x\in\xR^d. 
\end{equation} 
Moreover, the fact that the second inequality in \eqref{eq: step 2}, resulting from Jensen's inequality, is in fact an equality and \eqref{eq: step 4} yield $\Delta_{x_1} = \delta_{v_t(x_1)}\otimes\delta_{w_t(x_1)}$ for $\mu_t$-a.e. $x_1\in\xR^d$ (cf. a canonical proof of Jensen's inequality), i.e.
\begin{equation}\label{eq: Jensen}
\int_Y{\phi((x,r))\xdif \Delta_{x_1}} = \phi(v_t(x_1), w_t(x_1)) 
\end{equation}
for all $\phi\in\xCzero_b(Y)$, for $\mu_t$-a.e. $x_1\in\xR^d$.

Altogether, we may conclude that $\Delta=((x,1),(v_t(x),w_t(x)))_{\#}\mu_t$,
\begin{equation}\label{eq: convergence of moments}
\int_{Y\times Y}{(\Lambda|x_2|^2+\Sigma|r_2|^2)\xdif\Delta} = |\mu_t'|^2 =  \lim_{l\to\infty}\int_{Y\times Y}{(\Lambda|x_2|^2+\Sigma|r_2|^2)\xdif\Delta_{t,h_l,s}}
\end{equation}
and that \eqref{eq: convergence of delta} holds good for all $\Phi\in\xCzero_b(Y\times Y)$. A similar argument as in the proof of \eqref{eq: 3}, making use of \eqref{eq: convergence of moments}, will show \eqref{eq: convergence of delta} for all continuous functions $\Phi: Y\times Y \to \xR$ satisfying the growth condition \eqref{eq: growth condition} (cf. Thm. 7.12 in \cite{Villani03} where the space of Borel probability measures with finite second order moments is considered and the equivalence between convergence in the Kantorovich-Wasserstein distance and convergence in duality with continuous functions satisfying a suitable growth condition is proved). This completes the proof of Prop. \ref{prop: push-forward}(i). \\

(ii) Let $t\in\mathcal{N}$. According to \eqref{eq: cone 1}, \eqref{eq: Wasserstein on the cone}, we have
\begin{equation}\label{eq: ii 1}
\frac{\HK_{\Lambda, \Sigma}(\mu_{t+h},\mathfrak{h}\chi_{t,h})^2}{h^2} \ \leq \ \frac{1}{h^2}\int_{\mathfrak{C}\times\mathfrak{C}}{{\sf d}_{\mathfrak{C}, \Lambda, \Sigma}(\Xi_{t,h}([x_1,r_1]), [x_2,r_2])^2 \xdif \beta_{t,t+h}}.
\end{equation}
We will prove that the right-hand side of \eqref{eq: ii 1} converges to $0$ as $h\to 0$. 

First we note that, by Prokhorov's Theorem, Def. \ref{def: N}(ii) and the proof of Prop. \ref{prop: push-forward}(i), every sequence $\Big(((v_t(x_1),w_t(x_1)), (x_2,r_2))_{\#}\Delta_{t,h_l,s}\Big)_{l\in\mathbb{N}}, \ h_l\to0,$ is relatively compact w.r.t. the weak topology in $\mathcal{M}(Y\times Y)$ and in duality with continuous functions $\Phi: Y\times Y\to\xR$ satisfying \eqref{eq: growth condition}, and the second marginals of the corresponding limit measures coincide with $(v_t(x), w_t(x))_{\#}\mu_t$. It follows from this and from an application of the Dominated Convergence Theorem that
\begin{eqnarray*}
\lim_{N\to\infty}\limsup_{h\to 0}\frac{1}{h^2}\int_{(\mathfrak{C}\setminus\mathfrak{C}_{t,1/N})\times\mathfrak{C}}{{\sf d}_{\mathfrak{C},\Lambda, \Sigma}([x_1,r_1],[x_2,r_2])^2\xdif\beta_{t,t+h}} \ = \ 0, 
\end{eqnarray*}
which implies 
\begin{equation}\label{eq: ii 2}
\lim_{h\to 0} \ \frac{1}{h^2}\int_{(\mathfrak{C}\setminus\mathfrak{C}_{t,h})\times \mathfrak{C}}{{\sf d}_{\mathfrak{C},\Lambda, \Sigma}([x_1,r_1],[x_2,r_2])^2\xdif \beta_{t,t+h}} \ = \ 0. 
\end{equation}

Next we consider 
$\frac{1}{h^2}\int_{\mathfrak{C}_{t,h}\times\mathfrak{C}}{{\sf d}_{\mathfrak{C}, \Lambda, \Sigma}(\Xi_{t,h}([x_1,r_1]), [x_2,r_2])^2 \xdif \beta_{t,t+h}}$.
According to (\cite{Burago-Burago-Ivanov01}, Sect. 3.6) and (\cite{liero2018optimal}, Sect. 8.1), the geometric cone $(\mathfrak{C},{\sf d}_{\mathfrak{C},\Lambda, \Sigma})$ is a length space and it holds that any curve $\eta:=[x,r]: [0,1]\to\mathfrak{C}$ for $\xCone$-functions $x: [0,1]\to\xR^d$ and $r:[0,1]\to[0,+\infty)$ is absolutely continuous in $(\mathfrak{C},{\sf d}_{\mathfrak{C},\Lambda,\Sigma})$ and 
\begin{equation*}
{\sf d}_{\mathfrak{C},\Lambda, \Sigma}(\eta(1), \eta(0))^2\leq \int_0^1{\Big(\frac{4}{\Sigma}(r'(s))^2+\frac{1}{\Lambda} r(s)^2|x'(s)|^2\Big)\xdif s}
\end{equation*}
(cf. (\cite{liero2018optimal}, Lem. 8.1)). We define, for $y_1:=[x_1,r_1]\in\mathfrak{C}_{t,h}$, $y_2:=[x_2,r_2]\in\mathfrak{C}$, with $|x_1-x_2|\leq \pi\sqrt{\Lambda/\Sigma}$ if $r_2 >0$, an absolutely continuous curve $\mathcal{C}_{h,\Xi(y_1),y_2}: [0,1]\to\mathfrak{C}$ connecting $\Xi(y_1)=[x_1+\Lambda hv_t(x_1), r_1(1+ \Sigma h w_t(x_1)/2)]$ and $y_2$ by setting $\mathcal{C}_{h,\Xi(y_1),y_2}:= [\mathcal{X}_{h,\Xi(y_1),y_2},\mathcal{R}_{h,\Xi(y_1),y_2}]$, 
\begin{eqnarray}
\mathcal{X}_{h,\Xi(y_1),y_2}(s) &:=& x_1 + \theta_{y_1,y_2}(s)(x_2-x_1) + \Lambda(1-s)h v_t(x_1),\\
\mathcal{R}_{h, \Xi(y_1),y_2}(s) &:=& \mathcal{R}_{y_1,y_2}(s)\Big(1+\Sigma(1-s)h w_t(x_1)/2\Big)
\end{eqnarray}
(cf. \eqref{eq: geodesics for push-forward}, \eqref{eq: transport}). The functions $\mathcal{X}_{h,\Xi(y_1),y_2}: [0,1]\to\xR^d$ and $\mathcal{R}_{h,\Xi(y_1),y_2}: [0,1]\to [0,+\infty)$ are continuously differentiable with
\begin{eqnarray*}
(\mathcal{R}_{h,\Xi(y_1),y_2}'(s))^2 &=& \Big(\Sigma\mathcal{R}'_{y_1,y_2}(s)(1-s)h w_t(x_1)/2 + \mathcal{R}_{y_1,y_2}'(s) - \Sigma\mathcal{R}_{y_1,y_2}(s) hw_t(x_1)/2\Big)^2 \\ &\leq& 2|h| \Sigma \ {\sf d}_{\mathfrak{C},\Lambda, \Sigma}(y_1,y_2)^2 + 2\Big(\mathcal{R}'_{y_1,y_2}(s)-\Sigma r_1 hw_t(x_1)/2\Big)^2
\end{eqnarray*}
and 
\begin{eqnarray*}
\mathcal{R}_{h,\Xi(y_1),y_2}(s)^2|\mathcal{X}_{h,\Xi(y_1),y_2}'(s)|^2 \leq 4\mathcal{R}_{y_1,y_2}(s)^2|\theta'_{y_1,y_2}(s)(x_2-x_1) - \Lambda hv_t(x_1)|^2 \\
\leq 8\Big(|\mathcal{R}_{y_1,y_2}(s)\theta'_{y_1,y_2}(s)(x_2 - x_1)- \Lambda r_1 h v_t(x_1)|^2 + \Lambda^2 |h| |\mathcal{R}_{y_1,y_2}(s)-r_1|^2\Big) \\
\leq 8\Big(|\mathcal{R}_{y_1,y_2}(s)\theta'_{y_1,y_2}(s)(x_2 - x_1)- \Lambda r_1 h v_t(x_1)|^2 + \Lambda^2 \Sigma |h| / 4 \ {\sf d}_{\mathfrak{C},\Lambda, \Sigma}(y_1,y_2)^2\Big),
\end{eqnarray*}
where we have made use of \eqref{eq: geodesics for push-forward}, \eqref{eq: first derivatives R and theta} and the fact that $y_1=[x_1,r_1]\in\mathfrak{C}_{t,h}$. It follows from the above estimations and an application of Fubini's Theorem that
\begin{eqnarray*}
\frac{1}{h^2}\int_{\mathfrak{C}_{t,h}\times\mathfrak{C}}{{\sf d}_{\mathfrak{C}, \Lambda, \Sigma}(\Xi_{t,h}([x_1,r_1]), [x_2,r_2])^2 \xdif \beta_{t,t+h}} \\ \leq \frac{1}{h^2}\int_{\mathfrak{C}_{t,h}\times\mathfrak{C}}{\int_0^1{\Big(\frac{4}{\Sigma}(\mathcal{R}'_{h,\Xi(y_1),y_2}(s))^2+\frac{1}{\Lambda}\mathcal{R}_{h,\Xi(y_1),y_2}(s)^2|\mathcal{X}_{h,\Xi(y_1),y_2}'(s)|^2\Big)\xdif s}\xdif \beta_{t,t+h}} \\
\leq \int_0^1{\int_{Y\times Y}{\Big(2\Sigma(r_2-r_1 w_t(x_1))^2 + 8\Lambda|x_2-r_1 v_t(x_1)|^2\Big)\xdif \Delta_{t,h,s}((x_1,r_1),(x_2,r_2))}\xdif s} \\  + \ C_{\Lambda, \Sigma} \frac{\HK_{\Lambda, \Sigma}(\mu_t,\mu_{t+h})^2}{|h|}
\end{eqnarray*}
with $C_{\Lambda, \Sigma}$ only depending on $\Lambda$ and $\Sigma$. According to Def. \ref{def: N}(ii), there exists a sequence of functions $\zeta_n\in\xCinfty_c(\xR^d) \ (n\in\mathbb{N})$ so that $((\nabla\zeta_n, \zeta_n))_{n\in\mathbb{N}}$ converges to $(v_t,w_t)$ in $\xLtwo(\mu_t, \xR^d\times\xR)$, which means
\begin{equation}\label{eq: vorletzte}
\lim_{n\to\infty} \ \int_{Y\times Y}{\Big(r_1^2 (\zeta_n(x_1)-w_t(x_1))^2 + r_1^2|\nabla\zeta_n(x_1)-v_t(x_1)|^2\Big)\xdif \Delta_{t,h,s}((x_1,r_1),(x_2,r_2))}  = 0
\end{equation}
uniformly in $h\in(-t,1-t)$ and $s\in(0,1)$. Moreover, Prop. \ref{prop: push-forward}(i) and \eqref{eq: beta} yield 
\begin{equation}\label{eq: letzte}
\lim_{h\to0} \ \int_{Y\times Y}{\Big(\Sigma(r_2-r_1 \zeta_n(x_1))^2 + \Lambda|x_2-r_1 \nabla\zeta_n(x_1)|^2\Big)\xdif \Delta_{t,h,s}} = ||(v_t,w_t)-(\nabla\zeta_n,\zeta_n)||_{\xLtwo(\mu_t,\xR^d\times\xR)}^2 
\end{equation} 
for all $n\in\mathbb{N}$ and $s\in(0,1)$. Altogether, by applying Def. \ref{def: N}(i), \eqref{eq: vorletzte}, \eqref{eq: letzte} and Fatou's Lemma to the above estimation of $\frac{1}{h^2}\int_{\mathfrak{C}_{t,h}\times\mathfrak{C}}{{\sf d}_{\mathfrak{C}, \Lambda, \Sigma}(\Xi_{t,h}([x_1,r_1]), [x_2,r_2])^2 \xdif \beta_{t,t+h}}$, 
we obtain
\begin{equation}
\lim_{h\to0} \ \frac{1}{h^2}\int_{\mathfrak{C}_{t,h}\times\mathfrak{C}}{{\sf d}_{\mathfrak{C}, \Lambda, \Sigma}(\Xi_{t,h}([x_1,r_1]), [x_2,r_2])^2 \xdif \beta_{t,t+h}} \ = \ 0, 
\end{equation}
which completes the proof of Prop. \ref{prop: push-forward}(ii).  
\end{proof}

We are now in a position to compute the derivative \eqref{eq: derivatives haupt} at every $t\in\mathcal{N}$. 
\begin{theorem}\label{thm: main theorem} 
If $t\in\mathcal{N}$ and $\beta_{t, \star}\in\mathcal{M}(\mathfrak{C}\times\mathfrak{C})$ is optimal in the definition of $\HK_{\Lambda, \Sigma}(\mu_t, \nu)^2$ according to \eqref{eq: cone}, \eqref{eq: Wasserstein on the cone}, with first marginal $\alpha_t\in\mathcal{M}_2(\mathfrak{C}), \ \mathfrak{h}\alpha_t\leq \mu_t$, and second marginal $\alpha_\star\in\mathcal{M}_2(\mathfrak{C}), \ \mathfrak{h}\alpha_\star\leq \nu$, then the derivative $\frac{\xdif}{\xdif t} [\frac{1}{2}\HK_{\Lambda, \Sigma}(\mu_t, \nu)^2]$ of \eqref{eq: 50} at $t$ coincides with
\begin{equation}\label{eq: 51}
\mathcal{F}_{t, \star} \ + \ 2\int_{\xR^d}{w_t(x)\xdif(\mu_t-\mathfrak{h}\alpha_t)}
\end{equation}
where $\mathcal{F}_{t, \star}$ is defined as
\begin{equation}\label{eq: 52} 2\int_{\mathfrak{C}\times\mathfrak{C}}{\Big[r_1^2 w_t(x_1)-r_1r_2 w_t(x_1)\cos(\sqrt{\Sigma/4\Lambda}|x_1-x_2|)-r_1r_2\sqrt{\Lambda/\Sigma} \ \langle S_{\Lambda,\Sigma}(x_1,x_2), v_t(x_1)\rangle\Big]\xdif\beta_{t,\star}},  
\end{equation}
with \begin{equation}\label{eq: SLambdaSigma}
S_{\Lambda, \Sigma}(x_1, x_2):= \begin{cases}
\frac{\sin(\sqrt{\Sigma/4\Lambda}|x_1-x_2|)}{|x_1-x_2|}(x_2-x_1) &\text{ if } x_1\neq x_2, \\
0 &\text{ if } x_1=x_2.
\end{cases}  
\end{equation}  
\end{theorem}
\begin{proof}
Let $t\in\mathcal{N}$. Then \eqref{eq: 50} is differentiable at $t$ and, by \eqref{eq: Xi},
\begin{equation}\label{eq: limit exists}
\frac{\xdif}{\xdif s} \Big[\frac{1}{2}\HK_{\Lambda, \Sigma}(\mu_s, \nu)^2\Big]\Bigg\vert_{s=t} \ = \ \lim_{h\to 0}\frac{\frac{1}{2}\HK_{\Lambda, \Sigma}(\mathfrak{h}\chi_{t,h}, \nu)^2 - \frac{1}{2}\HK_{\Lambda, \Sigma}(\mu_t, \nu)^2}{h},
\end{equation}
with $\chi_{t,h}$ defined as in Prop. \ref{prop: push-forward}(ii). Let $\bar{\chi}_{t,h}:=(\Xi_{t,h})_{\#}\alpha_t$ be the push-forward of $\alpha_t$ through the mapping $\Xi_{t,h}$ defined as in \eqref{eq: def of Xi}. We have 
\begin{eqnarray*}
\int_{\xR^d}{\phi\xdif(\mathfrak{h}\bar{\chi}_{t,h})} = \int_{\mathfrak{C}}{{\sf r}^2\phi({\sf x})\xdif\bar{\chi}_{t,h}}  = \int_{\mathfrak{C_{t,h}}}{{\sf r}^2(1+\Sigma hw_t({\sf x})/2)^2\phi({\sf x}+\Lambda hv_t({\sf x}))\xdif\alpha_t} + \int_{\mathfrak{C}\setminus\mathfrak{C}_{t,h}}{{\sf r}^2\phi({\sf x})\xdif \alpha_t}  
\\
= \int_{{\sf x}(\mathfrak{C}_{t,h})}{(1+\Sigma hw_t(x)/2)^2\phi(x+\Lambda hv_t(x))\xdif\mathfrak{h}\alpha_t} + \int_{{\sf x}(\mathfrak{C}\setminus\mathfrak{C}_{t,h})}{\phi(x)\xdif \mathfrak{h}\alpha_t} \\
\leq \int_{{\sf x}(\mathfrak{C}_{t,h})}{(1+\Sigma hw_t(x)/2)^2\phi(x+\Lambda hv_t(x))\xdif\mu_t} + \int_{{\sf x}(\mathfrak{C}\setminus\mathfrak{C}_{t,h})}{\phi(x)\xdif \mu_t} \ =  \int_{\xR^d}{\phi\xdif (\mathfrak{h}\chi_{t,h})} 
\end{eqnarray*} 
for all nonnegative bounded Borel functions $\phi: \xR^d\to \xR$ (cf. \eqref{eq: homogenous marginal}, \eqref{eq: homogenous marginal 2}), from which we infer that
\begin{eqnarray*}
\mathfrak{h}\bar{\chi}_{t,h} &\leq& \mathfrak{h}\chi_{t,h}, \\
(\mathfrak{h}\chi_{t,h} - \mathfrak{h}\bar{\chi}_{t,h})(\xR^d) &=& (\mu_t-\mathfrak{h}\alpha_t)(\xR^d) \ + \int_{{\sf x}(\mathfrak{C}_{t,h})}{\Big(\Sigma hw_t(x)+\frac{\Sigma^2}{4}h^2 w_t(x)^2\Big)\xdif (\mu_t-\mathfrak{h}\alpha_t)} . 
\end{eqnarray*} 
We obtain 
\begin{eqnarray*}
\frac{1}{2}\Big(\HK_{\Lambda, \Sigma}(\mathfrak{h}\chi_{t,h}, \nu)^2 - \HK_{\Lambda, \Sigma}(\mu_t, \nu)^2\Big)   &\leq&  \frac{1}{2}\Big(\mathcal{W}_{\mathfrak{C}, \Lambda, \Sigma}(\bar{\chi}_{t,h}, \alpha_\star)^2 - \mathcal{W}_{\mathfrak{C}, \Lambda, \Sigma}(\alpha_t, \alpha_\star)^2\Big) \\ &+& 2 \int_{{\sf x}(\mathfrak{C}_{t,h})}{\Big(hw_t(x)+\frac{\Sigma}{4}h^2 w_t(x)^2\Big)\xdif (\mu_t-\mathfrak{h}\alpha_t)}.
\end{eqnarray*}
The same argument as in the proof of Lem. 2.2 in \cite{fleissnerMMHK} yields 
\begin{eqnarray*}
\limsup_{h\downarrow 0}\frac{\frac{1}{2}\mathcal{W}_{\mathfrak{C}, \Lambda, \Sigma}(\bar{\chi}_{t,h}, \alpha_\star)^2 - \frac{1}{2}\mathcal{W}_{\mathfrak{C}, \Lambda, \Sigma}(\alpha_t, \alpha_\star)^2}{h} \leq \\ 2 \int_{\mathfrak{C}\times\mathfrak{C}}{\Big[r_1^2 w_t(x_1)-r_1r_2 w_t(x_1)\cos(\sqrt{\Sigma/4\Lambda}|x_1-x_2|)-r_1r_2\sqrt{\Lambda/\Sigma} \ \langle S_{\Lambda,\Sigma}(x_1,x_2), v_t(x_1)\rangle\Big]\xdif\beta_{t,\star}} \\ \leq \liminf_{h\uparrow 0}\frac{\frac{1}{2}\mathcal{W}_{\mathfrak{C}, \Lambda, \Sigma}(\bar{\chi}_{t,h}, \alpha_\star)^2 - \frac{1}{2}\mathcal{W}_{\mathfrak{C}, \Lambda, \Sigma}(\alpha_t, \alpha_\star)^2}{h}, 
\end{eqnarray*}
with $S_{\Lambda, \Sigma}$ defined as in \eqref{eq: SLambdaSigma}. Since the limit \eqref{eq: limit exists} exists and
\begin{equation*}
\lim_{h\to0} \int_{{\sf x}(\mathfrak{C}_{t,h})}{\Big(w_t(x)+\frac{\Sigma}{4}h w_t(x)\Big)\xdif (\mu_t-\mathfrak{h}\alpha_t)} \ = \  \int_{\xR^d}{w_t(x)\xdif (\mu_t-\mathfrak{h}\alpha_t)},
\end{equation*}
it follows from the above computations that
\begin{equation*}
\lim_{h\to 0}\frac{\frac{1}{2}\HK_{\Lambda, \Sigma}(\mathfrak{h}\chi_{t,h}, \nu)^2 - \frac{1}{2}\HK_{\Lambda, \Sigma}(\mu_t, \nu)^2}{h} \quad = \quad \mathcal{F}_{t,\star} \ + \ 2 \int_{\xR^d}{w_t(x)\xdif (\mu_t-\mathfrak{h}\alpha_t)}. 
\end{equation*}
The proof of Thm. \ref{thm: main theorem} is complete. 
\end{proof}
We would like to remark that the derivatives of \eqref{eq: 50} at $t\in\mathcal{N}$ can be expressed equally in terms of the Logarithmic Entropy-Transport characterization \eqref{eq: LET} of the Hellinger-Kantorovich distance $\HK_{\Lambda, \Sigma}$, by applying \eqref{eq: induced solution on the cone} to the above representation \eqref{eq: 51}, \eqref{eq: 52} of the derivatives. 

\paragraph{Acknowledgement} I gratefully acknowledge support from the Erwin Schr\"odinger International Institute for Mathematics and Physics (Vienna) during my participation in the programme ``Optimal Transport''.
\bibliographystyle{siam}
\bibliography{bibliografia}
\end{document}